\documentclass[12pt]{amsart}

\usepackage{amsmath}
\usepackage{latexsym}
\usepackage{amssymb}
\usepackage{amsfonts}

\usepackage{graphics}
\include{PDF}
\usepackage{caption,amsmath,amssymb,amsfonts,amsthm,latexsym,graphicx,multirow,color,hyperref,enumerate}

\def\qed{\ifmmode\square\else\nolinebreak\hfill
$\Box$\fi\par\vskip12pt}

\def\l{\langle} \def\r{\rangle} 
\def\div{\,{\Huge|}\,}

\def\C{{\rm C}}

\def\D{{\rm D}} \def\Q{{\rm Q}}
\def\S{{\rm S}} 
\def\J{{\rm J}}

\def\soc{{\sf soc}} 
\def\mod{{\sf mod~}}

\def\Aut{{\sf Aut}}

\def\PGammaL{{\rm P\Gamma L}} \def\PSigmaL{{\rm P\Sigma L}}
\def\A{{\rm A}}

\def\PSL{{\rm PSL}}\def\PGL{{\rm PGL}} 
 \def\SL{{\rm SL}}

\def\Ree{{\rm Ree}}

  \def\D{{\rm D}}

\newtheorem{theorem}{Theorem}
\newtheorem{lemma}[theorem]{Lemma}%
\newtheorem{corollary}[theorem]{Corollary}%

\begin{document}

\title{The exceptional Hall numbers}
\thanks{The project was partially supported by the NNSF of China (11931005, 12171223).}

\author{Zheng Guo}
	\address{Department of Mathematics \\
		Southern  University of Science and Technology,  Shenzhen 518055,  P.R.China}
	\email{12131227@mail.sustech.edu.cn}
\author{Yong Hu}
	\email{huy@sustech.edu.cn}
\author{Cai Heng Li}
	\address{Department of Mathematics \\
		SUSTech International Center for Mathematics\\
		Southern  University of Science and Technology,  Shenzhen 518055,  P.R.China}
	\email{lich@sustech.edu.cn}

\date\today

\maketitle

\begin{abstract}
A positive integer $m$ is called a {\it Hall number} if any finite group of order precisely divisible by $m$ has a Hall subgroup of order $m$. We prove that, except for the obvious examples, the three integers 12, 24 and 60 are the only Hall numbers, solving a problem proposed by Jiping Zhang.
\end{abstract}

\section{Introduction}

A divisor $m$ of an integer $n$ is called a {\it Hall divisor} if $n$ is {\it precisely divisible} by $m$, namely, $\gcd\left(m,{n\over m}\right)=1$.
A positive integer $m$ is a {\it Hall number} if any finite group of order having a Hall divisor $m$ has a Hall subgroup of order $m$.
Attempting to extend Sylow's theorem and Hall's theorem for solvable groups, Jiping Zhang proposes to determine Hall numbers.
By Sylow's theorem, a prime power is a Hall number, and Hall's theorem for finite solvable groups implies that $m$ is a Hall number if $m$ is even with $m/2$ odd.
The following theorem shows that, except for these two families of integers, 12, 24 and 60 are the only exceptional Hall numbers.
This solves the problem of Zhang regarding Hall numbers.

\begin{theorem}\label{thm:main}
A positive integer $m$ is a Hall number if and only if one of the following conditions holds:

\begin{itemize}
  \item[{\rm(1)}] $m$ is a prime power.
  \item[{\rm(2)}] $m$ is even with $m/2$ odd.
  \item[{\rm(3)}] $m=12$, $24$, or $60$.
\end{itemize}
\end{theorem}

The nonsolvable groups corresponding to the exceptional Hall numbers are classified below.

\begin{corollary}\label{cor:gps}
Let $G$ be a finite nonsolvable group which has a Hall subgroup $H$ of order $m$, where $m\in\{12,24,60\}$.
Then there exists a solvable subgroup $R\lhd G$ such that $\gcd(|R|,m)=1$ and $(m,H,G/R)$ lies in the following table, where $q=p^f$ with $p$ prime and $e|f$ with $\gcd(e,m)=1$.
\[\begin{array}{|llll|} \hline
m & H & G/R & \mbox{conditions} \\ \hline

12 & \A_4 & \PSL_2(q).\C_e & q^2\equiv 25, 121\ (\mod 144) \\ \hline

24 & \S_4 & \PGL_2(q).\C_e & q^2\equiv 25, 121\ (\mod 144) \\
  && \PSL_2(q).\C_e & q^2\equiv 49,241\ (\mod 288) \\ \hline
24 &\C_2\times\A_4 & \J_1 &\\
 &&\C_2\times\PSL_2(q).\C_e & q^2\equiv 25, 121\ (\mod 144) \\
 &\SL_2(3) & \SL_2(q).\C_e  & q^2\equiv 25, 121\ (\mod 144)\\ \hline

60 & \A_5 & \PSL_2(5)&\\
&&\PSL_2(q).\C_e & q^2\equiv 120k+1\ (\mod 3600) \\
 &&& \mbox{\small $k\in\{1,7,11,13,17,19,23,29\}$} \\
   & \C_5\times\A_4 & \C_5\times\PSL_2(q).\C_e & q^2\equiv 25, 121\ (\mod 144)\\
 &&\PSL_2(q).\C_{5e} & q^2\equiv 25, 121\ (\mod 144)\\ \hline
\end{array}\]
Moreover, in each line of the table, there are infinitely many choices for the values of $q$.
\end{corollary}

\section{A number theoretic lemma}

To prove Theorem~\ref{thm:main}, we need to establish a number theoretic lemma.

\begin{lemma}\label{lem:number}
 Let $a,b$ be integers such that $a,b>1$ and $\gcd(a,b)=1$.
 Then there exist infinitely many primes $p$ such that $a$ is a Hall divisor of $p-1$, and $b$ is a Hall divisor of $p+1$, namely,
 \[
 a\div(p-1)\,,\; b\div(p+1)\quad\text{ and } \quad\gcd\bigg(a\,,\,{p-1\over a}\bigg)=\gcd\bigg(b\,,\,{p+1\over b}\bigg)=1.
\]
\end{lemma}

Our proof of this lemma relies on the famous Dirichlet density theorem.

If $A$ is a set of prime numbers, let $\delta(A)$ (when it exists) denote its {\it (Dirichlet) density} as defined in \cite[Section\;VI.4]{SerreGTM7}. The following facts can be easily verified using the definition:

\begin{itemize}
  \item If $A$ is a finite set of primes, then $\delta(A)=0$.

  \item If $A_1,\,A_2$ are two sets of primes such that $\delta(A_1),\delta(A_2)$ and $\delta(A_1\cap A_2)$ all exist, then $\delta(A_1\cup A_2)$ exists and
\[
\delta(A_1\cup A_2)=\delta(A_1)+\delta(A_2)-\delta(A_1\cap A_2)\,.
\] This fact easily generalizes to an inclusion-exclusion principle for the density of a finite union of sets of primes.
\end{itemize}

For any positive integer $n$ and any integer $\xi$ such that $\gcd(\xi,n)=1$, define
\[
\mathcal{P}_\xi[n]:=\{ p \text{ prime}\,|\, p\equiv \xi\; (\mod{n})\}.
\]By Dirichlet's density theorem (\cite[Section\;VI.4, Thm.\;2]{SerreGTM7}), we have $\delta(\mathcal{P}_\xi[n])=\frac{1}{\varphi(n)}$, where $\varphi$ denotes Euler's totient function.

\begin{lemma}\label{lem:density}
  Let $m\ge 2$ be a positive integer and let $\xi$ be an integer such that $\gcd(\xi,m)=1$. 
  Then there exist infinitely many primes $p$ such that
  $p\equiv \xi\;(\mod{m})$ and $\gcd\big(m,\,\frac{p-\xi}{m}\big)=1$.
\end{lemma}
\begin{proof}
Let $p_1,\cdots, p_r$ be all the distinct prime divisors of $m$.  It is sufficient to prove that the set $\mathcal{P}_\xi[m]\setminus\big(\bigcup^r_{i=1}\mathcal{P}_\xi[mp_i]\big)$ has a positive density. This is equivalent to the inequality
\[
\delta\bigg(\bigcup^r_{i=1}\mathcal{P}_\xi[mp_i]\bigg)<\delta(\mathcal{P}_\xi[m])=\frac{1}{\varphi(m)}.
\]
By the inclusion-exclusion principle and Dirichlet's density theorem, we have
\[
\begin{split}
  \delta\bigg(\bigcup^r_{i=1}\mathcal{P}_\xi[mp_i]\bigg)
  &=\sum_{\emptyset\neq I\subseteq\{1,\cdots, r\}}(-1)^{1+|I|}\delta\bigg(\bigcap_{i\in I}\mathcal{P}_\xi[mp_i]\bigg)\\
  &=\sum_{\emptyset\neq I\subseteq\{1,\cdots, r\}}(-1)^{1+|I|}\delta\bigg(\mathcal{P}_\xi\bigg[m\prod_{i\in I}p_i\bigg]\bigg)\\
  &=\sum_{\emptyset\neq I\subseteq\{1,\cdots, r\}}(-1)^{1+|I|}\frac{1}{\varphi(m\prod_{i\in I}p_i)},
\end{split}
\]
where $I$ runs over nonempty subsets of the index set $\{1,\cdots, r\}$. Since
\[
\frac{\varphi(m)}{\varphi(m\prod_{i\in I}p_i)}=\prod_{i\in I}\frac{1}{p_i}\,,
\]
we only need to prove
\[
\sum_{\emptyset\neq I\subseteq\{1,\cdots, r\}}(-1)^{1+|I|}\prod_{i\in I}\frac{1}{p_i}<1\,,
\]
or equivalently,
\[
\sum_{I\subseteq\{1,\cdots, r\}}(-1)^{|I|}\prod_{i\in I}\frac{1}{p_i}>0\,.
\]
The left-hand side of this last inequality is equal to the product $\prod^r_{i=1}\big(1-\frac{1}{p_i}\big)$.
So the desired inequality is clear. The lemma is thus proved.\qed
\end{proof}

\medskip

\noindent {\it Proof of Lemma\;$\ref{lem:number}$.} \ 
By the Chinese remainder theorem, there is an integer $a$ such that
\[
\xi\equiv 1+a\;(\mod{a^2})\quad\text{ and }\quad \xi\equiv -1+b\;(\mod{b^2})\,.
\]
Thus, $\xi=1+a+y_1a^2=-1+b+y_2b^2$ for some $y_1,y_2\in\mathbb{Z}$.

Put $m=(ab)^2$. It is clear that $\gcd(\xi,m)=1$. 
So, by Lemma\;\ref{lem:density}, there are infinitely many primes $p$ such that
$x_p:=\frac{p-\xi}{m}\in \mathbb{Z}$ and $\gcd(m,x_p)=1$. For any such $p$, we have
\[
\frac{p-1}{a}=1+y_1a+x_pab^2\quad\text{and}\quad \frac{p+1}{b}=1+y_2b+x_pba^2.
\]
So it is easily seen that $p$ has the desired properties. \qed

\section{Proof of Theorem~\ref{thm:main}}

The necessity in Theorem~\ref{thm:main} will be established by analysing Hall subgroups of the simple groups $\PSL_2(p)$.
Subgroups of $\PSL_2(p)$ with $p$ prime are explicitly known, refer to \cite{Low-dim-book}.

\begin{lemma}\label{lem:PSL(2,p)}
Let $T=\PSL_2(p)$, where $p\ge5$ is a prime, and let $H<T$.
Then either
\begin{itemize}
\item[(1)] $H\le \C_p{:}\C_{p-1\over2}$, $\D_{p-1}$, or $\D_{p+1}$, or
\item[(2)] $H=\A_4$, $\S_4$, or $\A_5$.
\end{itemize}
\end{lemma}

Lemmas~\ref{lem:number} and \ref{lem:PSL(2,p)} enable us to obtain the candidates for Hall numbers. 
As usual, for an integer $n$ and a prime divisor $p$ of $n$, let $n_p$ be the $p$ part of $n$, namely, $n_p$ is the highest power of $p$ in $n$ so that $\gcd\left(n_p,{n\over n_p}\right)=1$.
Then denote ${n\over n_p}$ by $n_{p'}$.


\begin{lemma}\label{lem:candidates}
Let $m=ab$ be a positive integer such that $a,b>2$ and $\gcd(a,b)=1$.
Assume that $m\notin\{12,24,60\}$.
Then there exists a group $G=\PSL_2(p)$ for some prime $p$ such that $m$ is a Hall divisor of $|G|$, and that $G$ does not have a Hall subgroup of order $m$.

In particular, any  Hall number $m$ satisfies one of the three conditions in Theorem\;$\ref{thm:main}$.
\end{lemma}

\begin{proof}
By Lemma~\ref{lem:number}, there exists a prime $p$ such that $a$ is a Hall divisor of $p-1$, and $b$ is a Hall divisor of $p+1$.
Let $T=\PSL_2(p)$.
Then $|T|={1\over2}p(p-1)(p+1)$.
Noticing that $\gcd\left(p-1,p+1\right)=2$, we conclude that $m$ divides $(p-1)(p+1)/2$, and so $m$ is a Hall divisor of the order $|T|$.

Suppose that $H$ is a Hall subgroup of $T$  of order $m$.
If $H\le\C_p{:}\C_{p-1\over2}$, then $|H|=m=ab$, where $a\div(p-1)_{2'}$ and $b=2$, which is a contradiction.
If $H\le \D_{p-1}$, then $m=ab$ is such that $a\div(p-1)_{2'}$ and $b=2$, which is a contradiction.
If $H\le \D_{p+1}$, then $m=ab$ is such that $a=2$ and $b\div(p+1)_{2'}$, which is a contradiction.

Finally, if $H=\A_4$, $\S_4$, or $\A_5$, as in part~(2) of Lemma~\ref{lem:number}, then $m=|H|\in\{12,24,60\}$, which is a contradiction.
\qed\end{proof}

In the remaining part of the paper, we need to prove that the numbers satisfying one of the three conditions in Theorem\;$\ref{thm:main}$ are indeed Hall numbers.
As pointed out before, each power of a prime is a Hall number by Sylow's theorem;
an even number $m$ with $m/2$ odd is a Hall number by Hall's theorem for solvable groups.
Thus, to complete the proof of Theorem\;$\ref{thm:main}$, we only need to prove that each of the three numbers 12, 24 and 60 is a Hall number.

Recall that, for a finite group $G$, the {\it solvable radical} of $G$ is the largest solvable normal subgroup of $G$.
Obviously, the solvable radical of a finite group is a characteristic subgroup.
A finite group $G$ is said to be {\it almost simple} if $T\lhd G\le\Aut(T)$ for some nonabelian simple group $T$.
In this case, $T$ is called the {\it socle} of $G$, denoted by $\soc(G)$.

\begin{lemma}\label{lem:reduction}
Let $m\in\{12,24,60\}$, and let $G$ be a nonsolvable group such that $m$ is a Hall divisor of $|G|$.
Let $R$ be the solvable radical of $G$.
Then $G/R$ is an almost simple group with socle $T=\soc(G/R)$ such that a Sylow $2$-subgroup of $T$ is abelian or a dihedral group $\D_8$.
\end{lemma}
\begin{proof}
Since $G$ is nonsolvable, a Sylow 2-subgroup of $G$ is of order at least 4, and hence by the assumption, a Sylow 2-subgroup of $G$ is of order 4 or 8.
Let $R$ be the solvable radical of $G$.
Then $G/R$ is nonsolvable, and each minimal normal subgroup of $G/R$ is nonabelian.
Since $|G|_2=4$ or 8, we have that $|G/R|_2=4$ or 8.
Suppose that $S,T$ are two distinct minimal normal subgroups of $G/R$.
Then $\l S,T\r=S\times T\lhd G/R$, and $|S\times T|_2\ge 4^2=16$, which is a contradiction.
Thus $G/R$ has a unique minimal normal subgroup $T$.
Arguing similarly shows that $T$ is a simple group, and thus $G/R$ is an almost simple group.
So $T=\soc(G/R)$ is a nonabelian simple group.

Let $T_2$ be a Sylow 2-subgroup of $T$.
Since $|T_2|=4$ or 8, either $T_2$ is abelian, or $T_2=\D_8$ or $\Q_8$.
If $T_2=\Q_8$, then $T$ is not a nonabelian simple group, which is not possible.
Thus $T_2$ is abelian or $\D_8$.
\qed
\end{proof}

We thus need some results regarding nonabelian simple groups with Sylow 2-subgroups being abelian or $\D_8$.
The following list of simple groups with abelian Sylow 2-subgroups can be found in Theorem~4.126 of \cite{Gorenstein}.

\begin{lemma}\label{lem:simple-gps}
Let $T$ be a nonabelian simple group which has abelian Sylow $2$-subgroups.
Then $T=\PSL_2(q)$ with $q\equiv 3,5$ $(\mod 8)$, or $\PSL_2(2^f)$, or $\J_1$, or $\Ree(3^f)$ with $f>1$.
\end{lemma}

Finite simple groups with dihedral Sylow 2-subgroups were classified by Gorenstein in \cite{Gorenstein-1964}, from which we obtain the following lemma.

\begin{lemma}\label{lem:simple-gps-2}
Let $T$ be a nonabelian simple group which has Sylow $2$-subgroups isomorphic to $\D_8$.
Then $T=\PSL_2(q)$ with $q\equiv 7,9$ $(\mod 16)$.
\end{lemma}

Now we are ready to complete the proof of Theorem\;$\ref{thm:main}$ by the next three lemmas.

\begin{lemma}\label{lem:12}
Let $G$ be a nonsolvable group such that $12$ is a Hall divisor of $|G|$.
Then, letting $R$ be the solvable radical of $G$, the following hold:
\begin{itemize}
\item[$(1)$] $G/R=\PSL_2(q).\C_e\le \PSigmaL_2(q)$, where $q^2\equiv 25,121$ $(\mod 144)$, $\gcd(e,6)=1$ and

\item[$(2)$] $G$ has a Hall subgroup $H\cong\A_4$.
\end{itemize}
In particular, $12$ is a Hall number.
\end{lemma}

\begin{proof}
By Lemma~\ref{lem:reduction}, $G/R$ is an almost simple group.
Let $T=\soc(G/R)$.
Since $|G|_2=4$, it follows that $|T|_2=4$, and hence Sylow 2-subgroups of $T$ are abelian of order 4.
By Lemma~\ref{lem:simple-gps}, we conclude that $T=\PSL_2(q)$ with $q=p^f\equiv 3,5$ $(\mod 8)$ or $T=\PSL_2(2^2)\cong \A_5\cong \PSL_2(5)$. 
If $T=\PSL_2(4)$, then $G/R$ has a subgroup $\A_4$. 

Now we assume $T=\PSL_2(q)$ with $q=p^f\equiv 3,5$ $(\mod 8)$.
It follows that $f$ is odd, and $G/R=\PSL_2(q).\C_e\le\PSigmaL_2(q)$ with $e\div f$.
Since $12$ is a Hall divisor of $|\PSL_2(q)|=q(q^2-1)/2$ and $p$ is an odd prime, it follows either $12$ is a Hall divisor of $(q^2-1)/2$ or $q=3$.
Since $\PSL_2(q)$ is simple, $q>3$, and thus one can deduce that
\[\begin{array}{rcl}
&&\mbox{12 is a Hall divisor of ${q^2-1\over2}$}\\
 &\Longleftrightarrow& \mbox{24 is a Hall divisor of ${q^2-1}$}\\
 &\Longleftrightarrow & q^2-1= 24 x, \mbox{and $\gcd(x,6)=1$}\\
 &\Longleftrightarrow & q^2-1= 24 x, \mbox{and $x=6y+z$, with $z<6$ and $\gcd(z,6)=1$}\\
 &\Longleftrightarrow & q^2-1= 24 x=144y+24z, \mbox{with $z\in\{1,5\}$}\\
 &\Longleftrightarrow & q^2\equiv 24z+1\ (\mod 144),\ \mbox{with $z\in\{1,5\}$}.
 \end{array}\]
So $q^2\equiv 25,121$ $(\mod 144)$, as in part~(1).

Inspecting subgroups of $\PSL_2(q)$ given in \cite{Low-dim-book}, we conclude that $G/R$ indeed has a subgroup which is isomorphic to $\A_4$.
Now $\gcd(|R|,12)=1$, and thus $G=R.(G/R)$ has a subgroup $H$ that is isomorphic to $\A_4$, which is a Hall $\{2,3\}$-subgroup of $G$.
So $12$ is a Hall number.
\qed\end{proof}

The next lemma deals with the candidate $m=60$.

\begin{lemma}\label{lem:60}
Let $G$ be a nonsolvable group such that $60$ is a Hall divisor of $|G|$.
Then, letting $R$ be the solvable radical of $G$, one of the following holds:
\begin{enumerate}[{\rm(i)}]
\item $G/R=\PSL_2(q).\C_e\le \PSigmaL_2(q)$, and $G$ has a Hall subgroup $H\cong\A_5$, where $q^2\equiv 120k+1\ (\mod 3600)$ with $k\in\{1,7,11,13,17,19,23,29\}$ and $\gcd(e,30)=1$, or

\item $G/R=\PSL_2(q).\C_{5e}\le \PSigmaL_2(q)$, and $G/R$ has a Hall subgroup $H\cong\A_4\times\C_5$, where $q^2\equiv 25, 121\ (\mod 144)$ and $\gcd(e,30)=1$ or

\item $G/R_{5'}=\C_5\times\PSL_2(q).\C_e\le \C_5\times\PSigmaL_2(q)$, and $G$ has a Hall subgroup $H\cong\C_5\times\A_4$, where $q^2\equiv 25, 121\ (\mod 144)$ and $\gcd(e,30)=1$ or

\item $G=R:\PSL_2(5)$, and $G$ has a Hall subgroup $H=\PSL_2(5)\cong\A_5$.
\end{enumerate}
In particular, $60$ is a Hall number.
\end{lemma}

\begin{proof}
Obviously, in this case, 12 is also a Hall divisor of $|G|$.
By Lemma~\ref{lem:12}, the group $G/R=\PSL_2(q).\C_e\le \PSigmaL_2(q)$, where $q^2\equiv 25, 121\ (\mod 144)$.
Let $T=\PSL_2(q)$ be the socle of $G/R$.

Assume first that $5$ divides $|T|=|\PSL_2(q)|$.
Since $60$ is a Hall divisor of $|G|$, it follows that $60$ is a Hall divisor of $|T|=q(q^2-1)/2$, and so either $60$ is a Hall divisor of $(q^2-1)/2$ or $q=5$. 
If $q=5$, then $|T|=60$ and $G/R=\PSL_2(5)\cong \A_5$, so $G=R{:}T=R{:}\A_5$, with $\gcd(|R|,60)=1$, as in part~(iv).
Now assume that $60$ is a Hall divisor of $(q^2-1)/2$. 
Then one can deduce that
\[\begin{array}{rcl}
&&\mbox{60 is a Hall divisor of ${q^2-1\over2}$}\\
 &\Longleftrightarrow& \mbox{120 is a Hall divisor of ${q^2-1}$}\\
 &\Longleftrightarrow & q^2-1= 120 x, \mbox{and $\gcd(x,30)=1$}\\
 &\Longleftrightarrow & q^2-1= 120 x, \mbox{and $x=30y+k$, with $k<30$ and $\gcd(k,30)=1$}\\
 &\Longleftrightarrow & q^2-1= 120 x=3600y+120k, \mbox{with $k\in\{1,7,11,13,17,19,23,29\}$}\\
 &\Longleftrightarrow & q^2\equiv 120k+1\ (\mod 3600),\ \mbox{with $k\in\{1,7,11,13,17,19,23,29\}$}.
 \end{array}\]
Furthermore, an inspection of subgroups of $T=\PSL_2(q)$ in \cite{Low-dim-book} shows that $T$ has a subgroup which is isomorphic to $\A_5$. Now $\gcd(|R|,60)=1$, and thus $G=R.(G/R)$ has a subgroup $H$ that is isomorphic to $\A_5$, as in part~(i).

Next, assume $5$ divides $|G/R|$ but does not divide $|T|$.
Then $q=p^f$ with $f$ divisible by 5, and $G/R\ge T.\l\varphi\r=\PSL_2(p^f).\C_5$, where $\varphi$ is a field automorphism of $T$ of order 5.
It follows that $\varphi$ centralizes $\PSL_2(p)$, and so $\varphi$ centralizes the subgroups of $\PSL_2(p)$.
Since $\PSL_2(p)$ contains a subgroup which is isomorphic to $\A_4$, we conclude that $G/R$ has a Hall subgroup which is isomorphic to $\A_4\times\C_5$.
Since $\gcd(|R|,60)=1$, we conclude that $G$ has a Hall subgroup which is isomorphic to $\A_4\times\C_5$, as in part~(ii).

Finally, assume that 5 does not divide $|G/R|$.
Then 5 divides $|R|$.
Since 12 divides $|T|$, we have that $\gcd(|R|,12)=1$ and $|R|_5=5$.
Thus $R$ is a 5-nilpotent by Burnside's criterion (see \cite[p.\,280]{Robinson}), and hence $R=R_{5'}{:}R_5=R_{5'}{:}\C_5$.
As $G/R\rhd T=\PSL_2(q)$, we obtain that $G/R_{5'}\rhd R_5.T=R_5\times\PSL_2(q)$.
Thus $G/R_{5'}$ has a subgroup which is isomorphic to $\C_5\times\A_4$.
Further, as $\gcd(|R_{5'}|,60)=1$, it follows that $G$ itself has a subgroup which is isomorphic to $\C_5\times\A_4$, as in part~(iii).

We thus conclude that $G$ has a Hall subgroup which is isomorphic to $\A_5$ or $\C_5\times\A_4$.
In particular, we conclude that $60$ is indeed a Hall number.
\qed\end{proof}

The final lemma treats the candidate $m=24$.

\begin{lemma}\label{lem:24}
Let $G$ be a nonsolvable group such that $24$ is a Hall divisor of $|G|$, and let $R$ be the solvable radical of $G$.
Then one of the following holds:
\begin{enumerate}[{\rm(i)}]

\item $G/R=\J_1$, and $G$ has a Hall subgroup $H\cong \C_2\times\A_4$.

\item $G/R=\PSL_2(q).\C_e\le \PSigmaL_2(q)$, and $G$ has a Hall subgroup $H\cong\S_4$, where $q^2\equiv 49, 241\ (\mod 288)$ and $\gcd(e,6)=1$, or

\item $G/R_{2'}=\C_2\times\PSL_2(q).\C_e$ or $\SL_2(q).\C_e$, and $G$ has a Hall subgroup $H\cong\C_2\times\A_4$ or $\SL_2(3)$, where $q^2\equiv 25, 121\ (\mod 144)$ and $\gcd(e,6)=1$, or

\item $G/R=\PGL_2(q).\C_e\le \PGammaL_2(q)$, and $G$ has a Hall subgroup $H\cong\S_4$, where $q^2\equiv 25, 121\ (\mod 144)$ and $\gcd(e,6)=1$.

\end{enumerate}
In particular, $24$ is a Hall number.
\end{lemma}

\begin{proof}
Since $24$ is a Hall divisor of $|G|$, a Sylow 2-subgroup of $G$ is of order 8.
Let $R$ be the solvable radical of $G$.
By Lemma~\ref{lem:reduction}, $G/R$ is an almost simple group, with socle $T$ say.
Then the order $|T|_2=4$ or 8.

First, assume that $|T|_2=8$.
Then by Lemmas~\ref{lem:simple-gps} and \ref{lem:simple-gps-2}, we have that
$T=\PSL_2(q)$ with $q\equiv 7,9$ $(\mod 16)$, or $\J_1$, or $\PSL_2(8)$, or $\Ree(q)$ with $q=3^f$.
Since $|G|$ is not divisible by 9, the order $|T|$ is not divisible by 9.
Thus $T\not=\PSL_2(8)$ or $\Ree(3^f)$ with $f>1$.
If $T$ is the sporadic simple group $\J_1$, then $T$ has a Hall subgroup which is isomorphic to $\C_2\times\A_4$ by the Atlas \cite{Atlas}, and so $G=R{:}J_1$ has a Hall subgroup $\C_2\times\A_4$, as in part~(i).
Now suppose that $T=\PSL_2(q)$.
Since $24$ is a Hall divisor of $|T|=|\PSL_2(q)|=q(q^2-1)/2$ and $q$ is odd, it follows that $24$ is a Hall divisor of $(q^2-1)/2$, and thus we have that
\[\begin{array}{rcl}
&&\mbox{24 is a Hall divisor of ${q^2-1\over2}$}\\
 &\Longleftrightarrow& \mbox{48 is a Hall divisor of ${q^2-1}$}\\
 &\Longleftrightarrow & q^2-1= 48 x, \mbox{and $\gcd(x,6)=1$}\\
 &\Longleftrightarrow & q^2-1= 48 x, \mbox{and $x=6y+z$, with $z<6$ and $\gcd(z,6)=1$}\\
 &\Longleftrightarrow & q^2-1= 48 x=288y+48z, \mbox{with $z\in\{1,5\}$}\\
 &\Longleftrightarrow & q^2\equiv 48z+1\ (\mod 288),\ \mbox{with $z\in\{1,5\}$}.
 \end{array}\]
So $q^2\equiv 49,241$ $(\mod 288)$.
Indeed, in this case, $T=\PSL_2(q)$ has a Hall subgroup which is isomorphic to $\S_4$.
Moreover, since $\gcd(|R|,24)=1$, it follows that the group $G$ has a Hall subgroup of order 24.
This is as in part~(ii).

Now assume that $|T|_2=4$.
Then $T=\PSL_2(q)$, and $12$ is a Hall divisor of $|T|$, where $q^2\equiv 25, 121\ (\mod 144)$ by Lemma~\ref{lem:simple-gps}, and $T$ has a Hall subgroup which is isomorphic to $\A_4$.
In this case, we have that $|G/R|_2=4$ or 8.

We first deal with the case where $|G/R|_2=4$.
Then 2 divides $|R|$.
Since $\gcd\left(24,{|G|\over24}\right)=1$, we have that $|R|_2=2$.
Thus $R$ is a 2-nilpotent by Burnside's criterion (see \cite[p.\,280]{Robinson}), and hence $R=R_{2'}{:}R_2=R_{2'}{:}\C_2$.
As $G/R=\PSL_2(q).\l\varphi\r$ where $|\varphi|$ is odd, we obtain that $G/R_{2'}=R_2.(\PSL_2(q).\l\varphi\r)$.
It follows that $G/R_{2'}=\SL_2(q).\l\varphi\r$ or $(R_2\times\PSL_2(q)).\l\varphi\r$.
Thus $G/R_{2'}$ has a subgroup which is isomorphic to $2.\A_4$, where $2.\A_4\cong\SL(2,3)$ or $\C_2\times\A_4$.
Further, as $\gcd(|R_{2'}|,24)=1$, it follows that $G$ itself has a subgroup which is isomorphic to $\SL(2,3)$ or $\C_2\times\A_4$, as in part~(iii).

Suppose next that $|G/R|=8$.
Since $|T|_2=4$, we conclude that $G/R\ge\PGL_2(q)$.
Checking the subgroups of $\PGL_2(q)$ described in \cite{Low-dim-book} shows that $\PGL_2(q)$ indeed has a subgroup which is isomorphic to $\S_4$.
Since $\gcd(|R|,24)=1$, it follows that $G$ has a Hall subgroup which is isomorphic to $\S_4$, as in part~(iv).

In summary, the group $G$ has a Hall subgroup that is isomorphic to $\C_2\times\A_4$, $\S_4$ or $\SL_2(3)$.
Thus, in particular, we obtain that $24$ is a Hall number.
\qed\end{proof}

{\bf Proofs of Theorem~\ref{thm:main} and Corollary~\ref{cor:gps}:}

Let $m$ be a Hall number.
To complete the proofs of Theorem~\ref{thm:main} and Corollary~\ref{cor:gps}, we may assume that $m=ab$ such that $a,b>2$ and $\gcd(a,b)=1$.
Then $m=12$, 24 or 60 by Lemma~\ref{lem:candidates}.

Conversely, let $G$ be a nonsolvable group such that $m$ is a Hall divisor of $|G|$.
Then by Lemmas~\ref{lem:12}, \ref{lem:60} and \ref{lem:24}, all of the three numbers are Hall numbers, and $G$ has a Hall subgroup of order $m$ such that $(m,H,G)$ is as in the table of Corollary~\ref{cor:gps}.

Finally, let $m=3b$, where $b=4$, 8 or 15.
By Lemma~\ref{lem:number}, there are infinitely many primes $p$ such that 3 is a Hall divisor of $p-1$ and $b$ is a Hall divisor of $p+1$.
It follows that, for each line in the table of Corollary~\ref{cor:gps}, there are infinitely many choices for $q$.
\qed

\vskip0.3in
\end{document}